\documentclass[12pt,reqno]{amsart}
\usepackage{amsmath, amssymb, amsthm}
\usepackage{url}
\usepackage[breaklinks]{hyperref}

\renewcommand{\Im}{\operatorname{Im}}

\renewcommand{\Re}{\operatorname{Re}}
\renewcommand{\Im}{\operatorname{Im}}

\renewcommand{\(}{\left\(}
\renewcommand{\)}{\right\)}
\renewcommand{\[}{\left\[}
\renewcommand{\]}{\right\]}

\numberwithin{equation}{section}
 \theoremstyle{plain}
\newtheorem{theorem}{Theorem}[section]
\newtheorem{lemma}[theorem]{Lemma}

\newtheorem{corollary}[theorem]{Corollary}

   \makeatletter
\def\proof{\@ifnextchar[{\@oproof}{\@nproof}}
\def\@oproof[#1][#2]{\trivlist\item[\hskip\labelsep\textit{#2 Proof of\
#1.}~]\ignorespaces}
\def\@nproof{\trivlist\item[\hskip\labelsep\textit{Proof.}~]\ignorespaces}

\makeatother

\begin{document}
\title[An asymptotic expansion for a Lambert series]{An asymptotic expansion for a Lambert series associated to the symmetric square $L$-function}

\author{Abhishek Juyal}
\address{Abhishek Juyal\\Department of Mathematics\\
The Institute of Mathematical Sciences \\
4th Cross Street, CIT Campus, Tharamani \\ 
Chennai, Tamil Nadu 600113, India } 
\email{abhinfo1402@gmail.com}

 \author{Bibekananda Maji}
\address{Bibekananda Maji\\ Discipline of Mathematics \\
Indian Institute of Technology Indore \\
Indore, Simrol, Madhya Pradesh 453552, India.} 
\email{bibekanandamaji@iiti.ac.in}

\author{Sumukha Sathyanarayana}
\address{Sumukha Sathyanarayana\\Department of Mathematical and Computational Sciences\\ National Institute of Technology Karnataka, Suratkal \\
Srinivasnagar, Mangalore 575025, Karnataka, India.} 
\email{neerugarsumukha@gmail.com}

\thanks{2010 \textit{Mathematics Subject Classification.} Primary 11M06; Secondary 11M26, 11N37.\\
\textit{Keywords and phrases.} Lambert series, Riemann zeta function, non-trivial zeros, Symmetric square $L$-function,  Rankin-Selberg $L$-function}

\maketitle

\begin{abstract}
Hafner and Stopple proved a conjecture of Zagier,  that the inverse Mellin transform of the symmetric square $L$-function associated to the Ramanujan tau function has an asymptotic expansion in terms of the non-trivial zeros of the Riemann zeta function $\zeta(s)$.  Later,  Chakraborty,  Kanemitsu and the second author extended this phenomenon for any Hecke eigenform over the full modular group.  In this paper,  we study an asymptotic expansion of the Lambert series
$$
y^k \sum_{n=1}^\infty \lambda_{f}( n^2 ) \exp (- ny),  \quad \textrm{as}\,\,  y \rightarrow 0^{+},
$$
where $\lambda_f(n)$ is the $n$th Fourier coefficient of a Hecke eigen form $f(z)$ of weight $k$ over the full modular group.  
\end{abstract}

\section{Introduction}
Let $\Delta(z):= e^{2\pi i z} \prod_{n=1}^{\infty} (1 - e^{2 \pi i n z})^{24 }= \sum_{n=1}^{\infty} \tau(n) e^{2 \pi i n z}$ be the Ramanujan cusp form of weight $12$. 
Around four decades ago, Don Zagier \cite[p.~417]{Zag}, \cite[p.~271]{Zag92} conjectured that the constant term of the automorphic form $\Im(z)^{12}| \Delta(z)|^2$, that is, the Lambert series $$ a_{0}(y):= y^{12}\sum_{n=1}^\infty \tau(n)^2 \exp(-4 \pi n y) $$ has an asymptotic expansion as $y\rightarrow 0^{+}$, and it can be written in terms of the non-trivial zeros of the Riemann zeta function $\zeta(s)$.  He also claimed that $a_{0}(y)$ will have an oscillatory behaviour when $y \rightarrow 0^{+}$. 
Moreover,  his prediction was that,  $a_{0}(y)$ will satisfy
\begin{align*}
a_0(y) \sim A + \sum_{\rho} y^{ 1 - \frac{\rho}{2}} B_{\rho}\quad \textrm{as}\,\,  y \rightarrow 0^{+}, 
\end{align*}
where $A$ is some constant and the sum over $\rho$ runs through all non-trivial zeros of $\zeta(s)$,  and $B_\rho$ are some complex numbers.  Under the assumption of the Riemann Hypothesis,  one can write the above asymptotic expansion as
\begin{align*}
a_0(y) \sim A + y^{3/4} \sum_{n=1}^\infty a_n \cos\left( \phi_n - \frac{t_n}{2} \log(y) \right) \quad \textrm{as}\,\,  y \rightarrow 0^{+}, 
\end{align*}
where $a_n$ and $\phi_n$ are some constants.  The oscillatory behaviour of $a_0(y)$ is due to the presence of cosine functions in the above asymptotic expansion.
This conjecture was finally settled by Hafner and Stople \cite{HS} in 2000.  In 2017, Chakraborty, Kanemitsu, and the second author  \cite{CKM}  extended this phenomenon for Hecke eigenform over the full modular group. Again, Chakraborty et al.  \cite{CJKM} also extended it for congruence subgroup. Mainly, they proved that the constant terms of the automorphic form $y^k |f(z)|^2$, where $f(z)$ is a Hecke eigenform of weight $k$ over the full modular group $SL(2,\mathbb{Z})$, that is, $$y^k \sum_{n=1}^{\infty} |\lambda_f(n)|^2 \exp(-4 \pi n y)$$ also has an asymptotic expansion as $y \rightarrow 0^{+}$, and it can be expressed in terms of the non-trivial zeros of the Riemann zeta function.  Recently,  Banerjee and Chakraborty \cite{BC-19} studied this phenomenon for the Maass cusp forms. 

In the present paper, we investigate an asymptotic expansion of the Lambert series 
$$ y^k \sum_{n=1}^{\infty} \lambda_f(n^2) \exp(-n y) $$ as $y \rightarrow 0^{+}$.
Interestingly,  we observe that the asymptotic expansion of this Lambert series can also be written in terms of the non-trivial zeros of the Riemann zeta function $\zeta(s)$.

\section{Preliminaries}
Let $k$ and $N$ be two positive integers and $\chi$ be a primitive Dirichlet character modulo $N$. We define
\begin{align*}
\Gamma_{0}(N):= \left\{  \left(\begin{array}{c c}  a & b \\
c & d  \end{array} \right) \in {\rm SL}_2(\mathbb{Z}) :    { c} \equiv 0 \pmod N  \right\}.
\end{align*}
We denote $ S_k( \Gamma_{0}(N), \chi)$ be the space of cusp forms of weight $k$, level $N$, and Nebentypus character $\chi$. Let $f(z) \in  S_k( \Gamma_{0}(N), \chi)$ be a normalized Hecke eigenform with the Fourier series expansion 
\begin{align}\label{Fourier}
f(z) = \sum_{n=1}^\infty \lambda_{f}(n) \exp(2 \pi i n z), \quad \forall z \in \mathbb{H}.
\end{align}
The $L$-function associated to $f(z)$ satisfy following Euler product representation:
\begin{align*}
L(s, f):= \sum_{n=1}^{\infty} \frac{\lambda_{f}(n)}{n^s} & =   \prod_{p:\, {\rm prime} } \left( 1 - \lambda_{f}(p) p^{-s} + \chi(p) p^{k-1-2s} \right)^{-1} \\
  & = \prod_{p:\, {\rm prime} } ( 1 - \alpha_{p} p^{-s})^{-1} (1 - \beta_{p} p^{-s})^{-1}, \quad \Re(s) > \frac{k+1}{2},
\end{align*}
where the complex conjugates $\alpha_p$ and $\beta_p$ satisfy the relations $\alpha_p+\beta_p=\lambda_f(p)$ and $\alpha_p \beta_p =\chi(p) p^{k-1}$. With the help of these complex numbers, Shimura \cite[Equation (0.2)]{Shm} introduced a new $L$-function associated to a Hecke eigenform $f(z)$, namely symmetric square $L$-function, which is defined by
\begin{align*}
L(s, {\rm Sym}^2(f) \otimes \psi):= \prod_{p:\, {\rm prime} } \left( 1 - \psi(p)\alpha_{p}^2 p^{-s}  \right)^{-1}\left( 1 - \psi(p)\beta_{p}^2 p^{-s}  \right)^{-1} \left( 1 - \psi(p)\alpha_{p} \beta_{p} p^{-s}  \right)^{-1},
\end{align*}
where $\psi$ is a primitive Dirichlet character modulo $M$. This is one of the important examples of an $L$-function associated to a ${\rm GL}(3)$-automorphic form and its analytic continuation and functional equation has been studied by Shimura. More generally,  we can define symmetric power $L$-function associated to $f(z)$ as follows:
\begin{align*}
L(s, {\rm Sym}^n(f) \otimes \psi) :=  \prod_{p:\, {\rm prime} } \prod_{i=0}^{n} \left( 1 - \psi(p) \alpha_{p}^i \beta_{p}^{n-i} \right)^{-1}. 
\end{align*}
Interested readers can see Murty's \cite{Murty} lecture notes for more information on symmetric power $L$-function.
Upon simplification of the Euler product of the symmetric square $L$-function, Shimura observed that 
\begin{align}\label{symmetric square_relation}
L(s, {\rm Sym}^2(f) \otimes \psi) = L(2s - 2 k + 2, \chi^2 \psi^2) \sum_{n=1}^{\infty} \frac{\lambda_{f}(n^2) \psi(n)}{n^{s}},
\end{align}
where $L(s,\chi)$ is the usual Dirichlet $L$-function.
In the same paper, Shimura established following important result:
\begin{theorem}
Let us	define 
	\begin{align*}
		L^{*}(s, {\rm Sym}^2(f) \otimes \psi):=N^s \pi^{-\frac{3s}{2}} \Gamma\left(\frac{s}{2}\right) \Gamma\left(\frac{s+1}{2} \right) \Gamma\left( \frac{s-k+2-\lambda_{0}}{2}  \right) L(s, {\rm Sym}^2(f) \otimes \psi),
	\end{align*}
	where $\lambda_0 = \begin{cases}
		0, &  \rm{if} \,\, \chi\psi(-1)=1,\\
		1. &  \rm{if}\,\, \chi\psi(-1)=-1.
		
	\end{cases}$
	 
Then $L^{*}(s, {\rm Sym}^2(f) \otimes \psi)$ can be analytically continued to the complex plane except for simple poles at $s=k$ and at $s=k-1$. 
\end{theorem}
Rankin \cite{Ran1} and Selberg \cite{Selberg} independently studied following interesting Dirichlet series associated to the cusp form $f(z)$, namely,
\begin{align*}
RS(s, f \otimes \bar{f}) := \sum_{n=1}^{\infty} \lambda^2_{f}(n) n^{-s}, \quad \Re(s)>k.
\end{align*}
This Dirichlet series is known as Rankin-Selberg $L$-function associated to $f(z)$. For a general construction of the Rankin-Selberg $L$-function,  readers can see the paper of Winnie Li \cite{Li}. The Rankin-Selberg $L$-function and the symmetric square $L$-function  are intimately connected with each other. This connection was established by Shimura, mainly, he observed that the following relation holds: 
\begin{align*}
L(s, {\rm Sym}^2(f) \otimes \psi) L(s-k+1, \chi \psi) = RS(s, f \otimes \bar{f}\otimes\psi) L(2s-2k+2,\chi^2\psi^2).
\end{align*}
For simplicity,  now onwards we assume $\chi$ and $\psi$ both are trivial characters. 
Thus,  the above relation becomes
\begin{align*}
L(s, {\rm Sym}^2(f)) \zeta(s-k+1) = RS(s, f \otimes \bar{f}) \zeta(2s-2k+2). 
\end{align*}
Since $\chi$ and $\psi$ both are trivial,  one can see that $\lambda_0=0$. In this case,
Shimura showed that the completed symmetric square $L$-function $L^{*}(s, {\rm Sym}^2(f))$ is entire and satisfy following beautiful functional equation:
\begin{align}\label{functional_symmetric square}
L^{*}(s, {\rm Sym}^2(f)) = L^{*}(2k-1-s, {\rm Sym}^2(f)).
\end{align}
This functional equation will be one of the crucial ingredients to obtain our main result.   The normalized version of the above functional equation can be found in \cite{IM}.

Now we introduce well-known hypergeometric series.
Let $a_1, a_2, \cdots, a_p$ and $b_1, b_2, \cdots, b_q$ be $p+q$ complex numbers. We define the generalized hypergeometric series  by
\begin{align}\label{hyper}
{}_pF_q \left( a_1, \cdots, a_p;\, b_1, \cdots, b_q;\, z \right):= \sum_{n=0}^{\infty} \frac{(a_1)_n \cdots (a_p)_n}{(b_1)_n\cdots (b_q)_n} \frac{z^n}{n!},
\end{align}
where $(a)_n:=\frac{\Gamma(a+n)}{\Gamma(a)}$ and $b_i$'s are not allowed to take non-positive integers.  This series converges for all complex values of $z$ if $p \leq q$, and when $p=q+1$ it converges for $|z|<1$, and in the latter case, Euler showed that it can be analytically continued to the entire complex plane except the branch cut from $1$ to $+\infty$. 



We define an arithmetic function $B_f(n)$,  connected with the symmetric square $L$-function by the relation:
\begin{align}\label{B_f(n)}
 B_f(n) := ( a_{{\rm Sym}^2(f)}*b)(n), \, \, {\rm where}\,\, b(n)= \begin{cases}  m^{2k-1}, &  \quad {\rm if }\quad n=m^2, \\
0, & \quad {\rm otherwise}.
  \end{cases}
  \end{align}
One can show that the Dirichlet series associated to $B_f(n)$ is absolutely convergent for $\Re(s)>k$. 
Now we are ready to state the main theorem.

\section{Main Results}

\begin{theorem}\label{main theorem}
Let $f(z) \in S_{k}(\mathrm{SL}_2(\mathbb{Z}))$ be a Hecke eigenform with the Fourier series expansion \eqref{Fourier}.  Assume all the non-trivial zeros of $\zeta(s)$ are simple.  Then for any positive real number $y$,  we have
\begin{align*}
 \sum_{n=1}^\infty  \lambda_f(n^2) \exp(-n y) = &  \frac{\Gamma(k)y^{1-k}}{2  \pi^2} \sum_{n=1}^{\infty}   \frac{B_f(n) }{n^k}
 \left[    {}_3F_2 \left( \frac{k}{2}, \frac{k+1}{2}, 1; \frac{1}{4}, \frac{3}{4}; -   \left( \frac{ y}{8 n \pi}  \right)^{2}    \right)- 1 \right] \nonumber \\
 & + \mathcal{R}(y),  
\end{align*}
where 
\begin{align*}
 \mathcal{R}(y) =  \frac{1}{2 y^{k-1} } \sum_{\rho} \frac{ \Gamma \left( \frac{\rho}{2} + k-1  \right) L(  \frac{\rho}{2} + k-1, {\rm Sym}^2(f))}{ y^{\frac{\rho}{2}} \zeta'(\rho) },
\end{align*}
and the sum over $\rho$ runs through all the non-trivial zeros of $\zeta(s)$ involves bracketing the terms so that the terms corresponding to $\rho_1$ and $\rho_2$ are included in the same bracket if they satisfy
	\begin{align*}
|\Im(\rho_1) - \Im(\rho_2)| < \exp \left( -C \frac{\Im(\rho_1)}{\log(\Im(\rho_1) )} \right) + \exp \left( -C \frac{\Im(\rho_2)}{\log(\Im(\rho_2) )} \right),
\end{align*}
where $C$ is some positive constant.

\end{theorem}

The below asymptotic result is an immediate application of the this theorem. 

\begin{corollary}\label{Cor1}
Let $N$ be a positive integer and $f(z)$ be a normalized Hecke eigenform defined as in Theorem \ref{main theorem}.  Assume Riemann Hypothesis and simplicity of the non-trivial zeros of $\zeta(s)$.  For $y \rightarrow0^{+}$,  we have
\begin{align*}
y^k  \sum_{n=1}^\infty  \lambda_f(n^2) \exp(-n y) = y^{3/4}  \sum_{n=1}^\infty  b_n \cos\left( \delta_n - \frac{t_n}{2} \log(y) \right) +  \sum_{j=1}^{N-1} A_j y^{2 j + 1} + O_{f, k}(y^{2N+1}),
\end{align*}
where the absolute constants $A_j$ depend only on $f$ and the polar representation of $\Gamma \left( \frac{\rho_n}{2} + k-1  \right) L(  \frac{\rho_n}{2} + k-1, {\rm Sym}^2(f)) \left(  \zeta'(\rho_n) \right)^{-1}$ is considered as $b_n \exp( i \delta_n)$,  where $\rho_n= \frac{1}{2} + i t_n$ denotes the nth non-trivial zero of $\zeta(s)$.  
\end{corollary}

\section{Some Well-known Results}
In this section,  we mention a few important well-known results that we will use frequently.  In his seminal paper, Riemann showed that $\zeta(s)$ can be analytically continued to the whole complex plane except for a simple pole at $s=1$ and satisfy following functional equation:
\begin{align}\label{func_zeta}
\pi^{-s/2}\Gamma\left( \frac{s}{2} \right) \zeta(s) = \pi^{-(1-s)/2} \Gamma\left( \frac{1-s}{2} \right) \zeta(1-s).
\end{align}

%
%
%
The gamma function $\Gamma(s)$ satisfy following duplication formula: 
\begin{lemma}
For any complex number $s$, we have 
\begin{equation}\label{duplication}
\Gamma(2s)=\frac{\Gamma(s)\Gamma(s+\frac{1}{2})\, 2^{2s}}{2\sqrt{\pi}}.
\end{equation}
\end{lemma}

The next result, gives an important asymptotic expansion for the gamma function,  known as Stirling's formula. 
\begin{lemma} \label{Starling}
In a vertical strip $c \leq \sigma \leq d$,  
\begin{equation}\label{starling1}
|\Gamma(\sigma + i T) | = \sqrt{2\pi} |T|^{\sigma- \frac{1}{2}} e^{-\frac{1}{2} \pi |T|} \left(1 + O\left(\frac{1}{|T|}\right)  \right), \quad {\rm as} \quad |T|\rightarrow \infty.
\end{equation}
\end{lemma}

\begin{lemma}\label{bound_zeta_inverse}
Let us suppose that there exist a sequence $T$ with arbitrary large absolute value that satisfy $\exp\left(-C_1 \Im(\rho)/\log(\Im(\rho) )   \right)< |T - \Im(\rho)| $ for every non-trivial zeros of $\zeta(s)$,  where $C_1$ is some positive constant.  Then 
\begin{align*}
\frac{1}{| \zeta(\rho + i T) |} < e^{C_2 T},
\end{align*}
where $0<C_2<\pi/4$ is some suitable  constant.
\end{lemma}
\begin{proof}
 A proof of this lemma can be seen in \cite[p.~219]{Titchmarsh}.
\end{proof}

\begin{lemma}\label{bound_symmetric}
In a vertical strip $\sigma_0 \leq \sigma \leq d$,  we have
\begin{align*}
|L(\sigma +i T, {\rm Sym}^2(f))| = O (|T|^{A(\sigma_0)}), \quad \mathrm{as}\,\, |T| \rightarrow \infty,
\end{align*}
where $A(\sigma_0)$ is some constant that depends on $\sigma_0$. 

\end{lemma}

\begin{proof}
One can find the proof of this lemma in \cite[p.~97]{IK}.
\end{proof}


Now we will introduce an important special function, namely,  Meijer $G$-function.  Let $m,n,p,q$ be  non-negative integers such that $0\leq m \leq q$, $0\leq n \leq p$.  Let $a_1, \cdots, a_p$ and $b_1, \cdots, b_q$ be $p+q$ complex numbers such that $a_i - b_j \not\in \mathbb{N}$ for $1 \leq i \leq n$ and $1 \leq j \leq m$.  The Meijer $G$-function  \cite[p.~415, Definition 16.17]{NIST}  is defined by the  line integral:
\begin{align}\label{MeijerG}
G_{p,q}^{\,m,n} \!\left(  \,\begin{matrix} a_1,\cdots , a_p \\ b_1, \cdots , b_q \end{matrix} \; \Big| z   \right) = \frac{1}{2 \pi i} \int_L \frac{\prod_{j=1}^m \Gamma(b_j - s) \prod_{j=1}^n \Gamma(1 - a_j +s) z^s  } {\prod_{j=m+1}^q \Gamma(1 - b_j + s) \prod_{j=n+1}^p \Gamma(a_j - s)}\mathrm{ds},
\end{align}
where the line of integration $L$ going from $-i \infty$ to $+i \infty$ and it separates the poles of the factors $\Gamma(b_j-s)$ from those of the factors $\Gamma(1-a_j+s)$.  The above integral converges if $p+q < 2(m+n)$ and $|\arg(z)| < (m+n - \frac{p+q}{2}) \pi$. Now we shall state Slater's theorem \cite[p.~415, Equation 16.17.2]{NIST}, which connects Meijer $G$-function with the  hypergeometric series by the following relation:  
If $p \leq q$ and $ b_j - b_k \not\in \mathbb{Z}$ for $j\neq k$, $1 \leq j, k \leq m$, then 
\begin{align}\label{MeijerG-Slater}
& G_{p,q}^{\,m,n} \!\left(   \,\begin{matrix} a_1, \cdots , a_p \\ b_1, \cdots , b_q \end{matrix} \; \Big| z   \right)  \\
& \quad = \sum_{k=1}^{m} A_{p,q,k}^{m,n}(z) {}_p F_{q-1} \left(  \begin{matrix}
1+b_k - a_1,\cdots, 1+ b_k - a_p \\
1+ b_k - b_1, \cdots, *, \cdots, 1 + b_k - b_q 
\end{matrix} \Big| (-1)^{p-m-n} z  \right), \nonumber 
\end{align}
where $*$ indicates that the entry $1 + b_k - b_k$ is omitted and 
\begin{align*}
A_{p,q,k}^{m,n}(z) := \frac{ z^{b_k}  \prod_{ j=1,  j\neq k}^{m} \Gamma(b_j - b_k ) \prod_{j=1}^n   \Gamma( 1 + b_k -a_j ) }{ \prod_{j=m+1}^{q} \Gamma(1 + b_k - b_{j}) \prod_{j=n+1}^{p} \Gamma(a_{j} - b_k)  }.
\end{align*}

\section{Proof of theorem \ref{main theorem}}
First, we show that the Mellin transform of the Lambert series $ \sum_{n=1}^\infty \psi(n) \lambda_f(n^2) \exp(-n y)$ is equals to $$ \frac{ \Gamma (s) L(s, {\rm Sym}^2(f) \otimes \psi)}{L(2s - 2 k + 2, \chi^2 \psi^2)} \quad \mathrm{for}\,\, \Re(s)>k. $$ 
That is, for $\Re(s) > k$, we write
\begin{align*}
\int_{0}^{\infty} \sum_{n=1}^\infty \psi(n) \lambda_f(n^2) \exp(-n y) y^{s-1} dy & =  \sum_{n=1}^\infty \psi(n) \lambda_f(n^2)  \int_{0}^{\infty} \exp(-n y) y^{s-1} \mathrm{d}y \\
& = \Gamma(s) \sum_{n=1}^\infty \psi(n) \lambda_f(n^2) n^{-s} \\
& = \frac{ \Gamma (s) L(s, {\rm Sym}^2(f) \otimes \psi)}{L(2s - 2 k + 2, \chi^2 \psi^2)},
\end{align*}
in the last step we have used the identity \eqref{symmetric square_relation}. Therefore, by inverse Mellin transform, we can see that, for $y >0$,   
\begin{align}\label{use_inverse_Mellin}
\sum_{n=1}^\infty \psi(n) \lambda_f(n^2) \exp(-n y) = \frac{1}{2 \pi i} \int_{c-i \infty}^{c+ i \infty} \frac{ \Gamma (s) L(s, {\rm Sym}^2(f) \otimes \psi)}{L(2s - 2 k + 2, \chi^2 \psi^2)} y^{-s} \mathrm{d}s,
\end{align}
where $\Re(s)=c >k$. We have already mentioned, for simplicity of the calculation, we assume $ \chi$ and $\psi$ are trivial characters. Thus, the above equation \eqref{use_inverse_Mellin} becomes
\begin{align}\label{right vertical line}
\sum_{n=1}^\infty  \lambda_f(n^2) \exp(-n y) = \frac{1}{2 \pi i} \int_{c-i \infty}^{c+ i \infty} \frac{ \Gamma (s) L(s, {\rm Sym}^2(f))}{\zeta(2s - 2 k + 2)} y^{-s} \mathrm{d}s. 
\end{align}
Now we shall analyse the poles of the integrand function.  Note that $\Gamma (s) L(s, {\rm Sym}^2(f))$ is an entire function since $L^{*}(s, {\rm Sym}^2(f))$ is entire as we are dealing with trivial character $\chi$.  In general,  $L^{*}(s, {\rm Sym}^2(f))$ may not be an entire function.  Assuming Riemann Hypothesis,  one can see that the integrand function has infinitely many poles on $\Re(s)= k-\frac{3}{4}$.  Furthermore,  the integrand function has  simple poles at $k-n$ for $n \geq 2$ due to the trivial zeros of  $\zeta(2s - 2 k + 2)$.  Consider the following rectangular contour $\mathcal{C}: [c-i T,   c+ i T],  [c+i T, d + i T],  [d +i T, d-i T],  $ and $[d-i T, c-i T]$,  where $k-2 <d < k-1$ and $T$ is a large positive real number.  We can observe that the integrand function has finitely many poles inside this contour $\mathcal{C}$ due to the non-trivial zeros $ \rho$ of $\zeta(2s - 2 k + 2)$ with $|\Im(\rho)| <T$	and the poles at $k-n$, for $n \geq 2$,  are lying outside the contour.  Therefore,  employing Cauchy residue theorem,  we have
\begin{align}\label{application_CRT}
\frac{1}{2 \pi i} \int_{\mathcal{C}} \frac{ \Gamma (s) L(s, {\rm Sym}^2(f))}{\zeta(2s - 2 k + 2)} y^{-s} \mathrm{d}s= \mathcal{R}_{T}(y),
\end{align}
where $\mathcal{R}_{T}(y)$ denotes the residual term that includes finitely many terms that are supplied by the non-trivial zeros $\rho$ of $\zeta(2s - 2 k + 2) $ with $|\Im(\rho)| < T$.  We denote two vertical integrals as 
\begin{align*}
V_1(T,  y):= & \frac{1}{2 \pi i} \int_{c-i T}^{c+ i T} \frac{ \Gamma (s) L(s, {\rm Sym}^2(f))}{\zeta(2s - 2 k + 2) } y^{-s} \mathrm{d}s,  \\
V_2(T, y):= & \frac{1}{2 \pi i} \int_{d-i T}^{d + i T} \frac{ \Gamma (s) L(s, {\rm Sym}^2(f))}{\zeta(2s - 2 k + 2)} y^{-s} \mathrm{d}s,
\end{align*}
and the horizontal integrals are denoted as
\begin{align*}
H_1(T, y):= & \frac{1}{2 \pi i} \int_{c+i T}^{d+ i T} \frac{ \Gamma (s) L(s, {\rm Sym}^2(f))}{\zeta(2s - 2 k + 2)} y^{-s} \mathrm{d}s, \\
H_2(T, y):= & \frac{1}{2 \pi i} \int_{d-i T}^{c- i T}  \frac{ \Gamma (s) L(s, {\rm Sym}^2(f))}{\zeta(2s - 2 k + 2)} y^{-s} \mathrm{d}s.
\end{align*}
Now one of the main aim is to show that the contribution of the horizontal integrals vanish as $T \rightarrow \infty$.  One can write
\begin{align*}
H_1(T, y) =  \frac{1}{2 \pi i} \int_{c}^{d}  \frac{ \Gamma ( \sigma + i T) L(\sigma +i T, {\rm Sym}^2(f))}{\zeta(2\sigma - 2 k + 2 + 2 i T)} y^{-\sigma-i T} \mathrm{d}\sigma.
\end{align*}
Thus,
\begin{align*}
|H_1(T, y)|  <\!\!<  \int_{c}^{d}  \frac{ | \Gamma ( \sigma + i T)| | L(\sigma +i T, {\rm Sym}^2(f))|}{|\zeta(2\sigma - 2 k + 2 + 2 i T)|} y^{-\sigma} \mathrm{d}\sigma.
\end{align*}

Use Lemmas \ref{Starling},  \ref{bound_zeta_inverse}, and \ref{bound_symmetric},  to  derive that 
\begin{align*}
|H_1( T, y)| <\!\!<   |T|^{C} \exp\left( (C_2 T- \frac{\pi}{4} |T| \right),
\end{align*}
where $C$ and $C_2$ are some constants with $0 <C_2 < \pi/4$.  This immediately implies that $H_1(T, y)$ goes to zero as $T \rightarrow \infty$. Similarly we can show that $H_2(T,y)$ also vanishes as $T \rightarrow \infty$.
Now allowing $ T \rightarrow \infty$ in $\eqref{application_CRT}$,  using \eqref{right vertical line},  we have
\begin{align}\label{residual+vertical}
\sum_{n=1}^\infty  \lambda_f(n^2) \exp(-n y) = \mathcal{R}(y)  + \frac{1}{ 2 \pi i} \int_{d- i \infty}^{d + i \infty}  \frac{ \Gamma (s) L(s, {\rm Sym}^2(f))}{\zeta(2s - 2 k + 2)} y^{-s} \mathrm{d}s,
\end{align}
where $\mathcal{R}(y)= \lim_{T \rightarrow \infty} \mathcal{R}_T(y)$ is the residual function consisting infinitely many terms.  
Assuming simplicity hypothesis,  that is,  all the non-trivial zeros of $\zeta(s)$ are simple,  one can show that 
\begin{align}
\mathcal{R}(y)  & =  \sum_{\rho}  \lim_{ s \rightarrow \frac{\rho}{2} + k-1}  \left( s -  \frac{\rho}{2} - k+1 \right)  \frac{ \Gamma (s) L(s, {\rm Sym}^2(f))}{\zeta(2s - 2 k + 2)} y^{-s}  \nonumber \\
 & =  \frac{1}{2 y^{k-1} } \sum_{\rho} \frac{ \Gamma \left( \frac{\rho}{2} + k-1  \right) L(  \frac{\rho}{2} + k-1, {\rm Sym}^2(f))}{\zeta'(\rho) y^{\frac{\rho}{2}}}, \label{final_residual}
\end{align}
where the sum over $\rho$ runs through all non-trivial zeros of $\zeta(s)$. \\
Now we shall try  to simplify the left vertical integral: 
\begin{align}\label{left vertical}
V_2 (y)= \lim_{T \rightarrow \infty} V_2(T,  y) = \frac{1}{ 2 \pi i} \int_{d- i \infty}^{d + i \infty}  \frac{ \Gamma (s) L(s, {\rm Sym}^2(f))}{\zeta(2s - 2 k + 2)} y^{-s} \mathrm{d}s.
\end{align}
First,  we shall make use of the functional equation of the symmetric square $L$-function.  Mainly,  employing \eqref{functional_symmetric square} and with the help of the duplication formula for the gamma function \eqref{duplication}, one can obtain
\begin{align} \label{vert_V2}
V_2(y) =  \frac{1}{2 \pi^{3k-1}}  \frac{1}{ 2 \pi i} \int_{d- i \infty}^{d + i \infty}  &  \frac{\Gamma  \left( \frac{2k-1}{2} - \frac{s}{2} \right) \Gamma  \left( k - \frac{s}{2} \right) \Gamma  \left( \frac{k+1}{2} - \frac{s}{2} \right)  }{  \Gamma  \left( \frac{2-k}{2} + \frac{s}{2} \right) \zeta(2s - 2 k + 2) } \nonumber  \\
& \quad \times L(2k-1-s, {\rm Sym}^2(f)) \left( \frac{y N^2}{2 \pi^3}  \right)^{-s} \mathrm{d}s. 
\end{align} 
Replace $s$ by $2s-2k+2$ in \eqref{func_zeta}  to see
\begin{align}\label{application_func_Dirichlet_L}
\zeta(2s-2k+2) =  \frac{ \pi^{2s-2k+2}}{\sqrt{\pi}}  \frac{\Gamma\left(  \frac{2k-2s -1}{2} \right)}{  \Gamma(1-k +s )} \zeta(2k -2s -1).
\end{align}
Substituting \eqref{application_func_Dirichlet_L} in \eqref{vert_V2} and simplifying,  we have
\begin{align*}
V_2(y)= \frac{1}{2  \pi^{k + \frac{1}{2}}  }  \frac{1}{ 2 \pi i} \int_{d- i \infty}^{d + i \infty}  &  \frac{\Gamma  \left( \frac{2k-1}{2} - \frac{s}{2} \right) \Gamma  \left( k - \frac{s}{2} \right) \Gamma  \left( \frac{k+1}{2} - \frac{s}{2} \right)  \Gamma(1-k +s ) }{  \Gamma  \left( \frac{2-k}{2} + \frac{s}{2} \right)  \Gamma \left( \frac{2k-1}{2} - s   \right) \zeta( 2 k -2s  -1) } \nonumber  \\
& \quad \times L(2k-1-s, {\rm Sym}^2(f)) \left( \frac{ y}{2 \pi}  \right)^{-s} \mathrm{d}s.  
\end{align*}
At this juncture, we would like to shift the line of integral and to do that we change the variable,  namely,  $2k-1-s =w$, then we obtain
\begin{align}\label{V_2(y)_after change of variable}
V_2(y)= \frac{1}{2  \pi^{k + \frac{1}{2}}  }  \frac{1}{ 2 \pi i} \int_{d'- i \infty}^{d' + i \infty} 
 & \frac{\Gamma\left( \frac{w}{2}  \right) \Gamma\left( \frac{w+1}{2}
  \right)  \Gamma\left( \frac{w}{2} + \frac{2-k}{2 }  \right) \Gamma(k-w)}{\Gamma\left( \frac{1+k}{2} - \frac{w}{2} \right)\Gamma\left(w + \frac{1-2k}{2} \right)} \nonumber \\
  & \quad \times \frac{L(w, {\rm Sym}^2(f))}{\zeta(2w -2k +1)}    \left( \frac{  y}{2 \pi}  \right)^{w-2k+1}  \mathrm{d}w,
\end{align}
where $ k < d'= \Re(w) < k+1$ as $ k-2 < d = \Re(s)< k-1$.  One can easily check that the symmetric square $L$-function $L(w,  {\rm Sym}^2(f))$ and $\zeta(2w -2k +1)$ are both absolutely convergent on the line $\Re(w)=d'$. Therefore,  we write
\begin{align}\label{generating_B}
\frac{L(w, {\rm Sym}^2(f))}{\zeta(2w -2k +1)}  & = \sum_{n=1}^\infty \frac{a_{Sym^2(f)}(n)}{n^w} \sum_{n=1}^\infty \frac{ n^{2k-1}}{n^{2w}} \nonumber \\
& = \sum_{n=1}^\infty \frac{B_f( n)}{n^w},  
\end{align}
where $B_f(n)$ is defined as in \eqref{B_f(n)}. 
Implement \eqref{generating_B} in \eqref{V_2(y)_after change of variable},  and  interchange the order of integration and summation to derive
\begin{align}\label{V_2(y) interms I(N,y)}
V_2(y)= \frac{1}{2  \pi^{k + \frac{1}{2}}  } \left( \frac{ y}{2 \pi}  \right)^{1-2k} \sum_{n=1}^{\infty} B_f(n)  I_{k, y}(n),
\end{align}
 where
 \begin{align*}
 I_{k,y}(n):= \frac{1}{ 2 \pi i} \int_{d'- i \infty}^{d' + i \infty}  
  \frac{\Gamma\left( \frac{w}{2}  \right) \Gamma\left( \frac{w+1}{2}
  \right)  \Gamma\left( \frac{w}{2} + \frac{2-k}{2 }  \right)    \Gamma(k-w)}{\Gamma\left( \frac{1+k}{2} - \frac{w}{2} \right)\Gamma\left(w + \frac{1-2k}{2} \right)} 
  \left( \frac{ y}{2 n \pi}  \right)^{w}  \mathrm{d}w.
 \end{align*}
 Now one of our main goals shall be to evaluate this line integral explicitly, if possible.  First, replace $w \rightarrow 2 w$,  
 \begin{align}\label{Integral_I(N,y)}
 I_{k,y}(n):= \frac{1}{ 2 \pi i} \int_{\frac{d'}{2}- i \infty}^{\frac{d'}{2} + i \infty}  
  \frac{\Gamma\left( w  \right) \Gamma\left(w + \frac{1}{2}
  \right)  \Gamma\left( w + \frac{2-k}{2 }  \right)    \Gamma(k-2w)}{\Gamma\left( \frac{1+k}{2} - w\right)\Gamma\left(2w + \frac{1-2k}{2} \right)} 
  \left( \frac{ y}{2 n \pi}  \right)^{2w} 2\,  \mathrm{d}w.
 \end{align}
 To simplify more we use duplication formula for the gamma function.  Mainly,  we have to use following two identities: 
 \begin{align}
 \Gamma( k - 2w)  & = \frac{2^{k-2w}}{2 \sqrt{\pi}} \Gamma \left( \frac{k}{2}-w \right) \Gamma \left(  \frac{1+k}{2}-w  \right), \label{duplication1} \\
 \Gamma\left(2w + \frac{1-2k}{2} \right) & = \frac{2^{2w + \frac{1-2k}{2}}}{2 \sqrt{\pi}} \Gamma \left( w +  \frac{1-2k}{4}  \right) \Gamma\left(w+ \frac{3-2k}{4}  \right). \label{duplication2}
 \end{align}
 Invoking  \eqref{duplication1} and \eqref{duplication2} in \eqref{Integral_I(N,y)} we arrive at
 \begin{align}
 I_{k,y}(n)&:= \frac{1}{ 2 \pi i} \int_{\frac{d'}{2}- i \infty}^{\frac{d'}{2} + i \infty}   \frac{\Gamma\left( w  \right) \Gamma\left(w + \frac{1}{2}
  \right)  \Gamma\left( w + \frac{2-k}{2 }  \right)    \Gamma \left( \frac{k}{2}-w \right) 2^{2k-4w-\frac{1}{2}} }{\Gamma \left( w +  \frac{1-2k}{4}  \right) \Gamma\left(w+ \frac{3-2k}{4}  \right)  } 
  \left( \frac{  y}{2 n \pi}  \right)^{2w} 2\,  \mathrm{d}w,   \nonumber \\
 & =\frac{2^{2k+\frac{1}{2}} }{ 2 \pi i} \int_{\frac{d'}{2}- i \infty}^{\frac{d'}{2} + i \infty}   \frac{\Gamma\left( w  \right) \Gamma\left(w + \frac{1}{2}
  \right)  \Gamma\left( w + \frac{2-k}{2 }  \right)    \Gamma \left( \frac{k}{2}-w \right) }{\Gamma \left( w +  \frac{1-2k}{4}  \right) \Gamma\left(w+ \frac{3-2k}{4}  \right)  } 
  \left( \frac{ y}{8 n \pi}  \right)^{2w}  \mathrm{d}w.  \label{Simplified_I(N,y)}
 \end{align}
To write this integral in terms of the Meijer $G$-function,  we shall analyse the poles of the integrand function.  We know the poles of $\Gamma(w)$ are at $0,-1,-2, \cdots$; poles of $\Gamma(w+ 1/2)$ are at $-1/2,-3/2,-5/2, \cdots$; and  the poles of $\Gamma\left( w + \frac{2-k}{2 }  \right) $ are at $k/2 -1, k/2-2,  k/2 -3, \cdots$; whereas the poles of $\Gamma \left( \frac{k}{2}-w \right) $ are at $k/2, k/2+1,  k/2 + 2, \cdots$.  Therefore,   we can not write the integral \eqref{Simplified_I(N,y)} in terms of the Meijer $G$-function  since the line of integration $\left( d'/2 \right) \in (k/2,  (k+1)/2)$ does not separate the poles of the gamma factors  $\Gamma\left( w  \right) \Gamma\left(w + \frac{1}{2}
  \right)  \Gamma\left( w + \frac{2-k}{2 }  \right) $ from the poles of the gamma factor $\Gamma \left( \frac{k}{2}-w \right)  $.   Hence,  we construct a new line of integration $(d'')$ with $d'' \in (k/2-1, k/2)$ so that it separates the poles of the gamma factors $\Gamma\left( w  \right) \Gamma\left(w + \frac{1}{2}
  \right)  \Gamma\left( w + \frac{2-k}{2 }  \right) $ from the poles of the gamma factor $\Gamma \left( \frac{k}{2}-w \right) $.  Now consider the contour $ \mathcal{C}'$  consisting of the line segments $[ d'-i T, d' + i T],  [d'+ i T,  d'' + i T],  [d'' + i T,  d'' - i T,]$,  and $[ d''- i T,  d'- iT]$,  where $T$ is some large positive real number, and employ Cauchy residue theorem to obtain
  \begin{align}\label{2nd_Appl_CRT}
  \frac{1}{2 \pi i} \int_{\mathcal{C}'  }   F_k(w)=   \mathop{\rm{Res}}_{s=\frac{k}{2}} F_{k}(w),
  \end{align}
 where 
 \begin{align*}
 F_k(w)= \frac{\Gamma\left( w  \right) \Gamma\left(w + \frac{1}{2}
  \right)  \Gamma\left( w + \frac{2-k}{2 }  \right)    \Gamma \left( \frac{k}{2}-w \right) }{\Gamma \left( w +  \frac{1-2k}{4}  \right) \Gamma\left(w+ \frac{3-2k}{4}  \right)  } 
  \left( \frac{ y}{8 n \pi}  \right)^{2w}  \mathrm{d}w.
  \end{align*}
Again,  with the help of Stirling's formula for the gamma function \eqref{starling1},   one can show that the horizontal integrals tend to zero as $T$ tends to infinity.  Therefore,  letting $T\rightarrow \infty$ in \eqref{2nd_Appl_CRT} and  calculating the residual term and substituting it in \eqref{Simplified_I(N,y)},  we will have 
\begin{align}\label{Final_I(N,y)}
I_{k,y}(n)= \frac{2^{2k+\frac{1}{2}} }{ 2 \pi i}  \int_{d''-i \infty }^{d''+i \infty}   F_k(w) \mathrm{d}w- \frac{  2^{2k+\frac{1}{2}} \Gamma\left( \frac{k}{2}  \right) \Gamma\left( \frac{k+1}{2}
  \right)  }{\Gamma \left( \frac{1}{4}  \right) \Gamma\left( \frac{3}{4}  \right)  } 
  \left( \frac{  y}{8 n \pi}  \right)^{k}. 
\end{align}
Now we shall try to write the line integral along $(d'')$ in-terms of the Meijer $G$-function and to do that we reminisce the definition of the Meijer $G$-function \eqref{MeijerG}.  We consider $m=1,  n= 3,   p=3,  q=3$  with $a_1=1,  a_2= 1/2,  a_3= k/2$; and $b_1= k/2, 
b_2= (1+2k)/4,  b_3 = (3+2k)/4$.   One can easily check that $a_i - b_j \not\in \mathbb{N} $ for $1\leq i \leq n,  1 \leq j \leq m$ and the inequality $p+q < 2(m+n)$ also satisfied.  Hence,  one can write
\begin{align}\label{appl_defn_Meijer}
\frac{1 }{ 2 \pi i}  \int_{d''-i \infty }^{d''+i \infty}   F_k(w) \mathrm{d}w=  G_{3,3}^{\,1,3} \!\left(  \,\begin{matrix} 1,\, \frac{1}{2},  \frac{k}{2} \\ \frac{k}{2},\, \frac{1+2k}{4},  \, \frac{3+2k}{4}     \end{matrix} \; \Big|  \left( \frac{ y}{8 n \pi}  \right)^{2}   \right).
\end{align}
Utilize Slater's theorem \eqref{MeijerG-Slater} to write the above Meijer $G$-function in-terms of the hypergeometric function: 
\begin{align}\label{appl_Slater}
G_{3,3}^{\,1,3} \!\left(  \,\begin{matrix} 1,\, \frac{1}{2},  \frac{k}{2} \\ \frac{k}{2},\, \frac{1+2k}{4},  \, \frac{3+2k}{4}     \end{matrix} \; \Big|  z  \right)= \frac{ z^{\frac{k}{2}   }\Gamma\left( \frac{k}{2}  \right) \Gamma\left( \frac{k+1}{2}
  \right)  }{\Gamma \left( \frac{1}{4}  \right) \Gamma\left( \frac{3}{4}  \right)  } 
   {}_3F_2 \left( \frac{k}{2}, \frac{k+1}{2}, 1; \frac{1}{4}, \frac{3}{4}; - z   \right). 
\end{align}
Substituting $z= \left( \frac{ y}{8 n \pi}  \right)^{2} $ in \eqref{appl_Slater} and together with \eqref{appl_defn_Meijer} and \eqref{Final_I(N,y)},  we achieve 
\begin{align}\label{Final_I(N,y)_hypergeom}
I_{k,y}(n) & =   \frac{  2^{2k+\frac{1}{2}} \Gamma\left( \frac{k}{2}  \right) \Gamma\left( \frac{k+1}{2}
  \right)  }{\Gamma \left( \frac{1}{4}  \right) \Gamma\left( \frac{3}{4}  \right)  } 
  \left( \frac{ y}{8 n \pi}  \right)^{k} \left[    {}_3F_2 \left( \frac{k}{2}, \frac{k+1}{2}, 1; \frac{1}{4}, \frac{3}{4}; -   \left( \frac{ y}{8 n \pi}  \right)^{2}    \right)- 1 \right].
\end{align}
Finally,  substituting $\eqref{Final_I(N,y)_hypergeom}$ in $ \eqref{V_2(y) interms I(N,y)}$ and  together with $\eqref{residual+vertical},$ $\eqref{final_residual}$, and $ \eqref{left vertical}$,  we complete the proof of Theorem \ref{main theorem}.  
%

\begin{proof}[Corollary  {\rm \ref{Cor1}}][]
With the help of the definition \eqref{hyper} of the hypergeometric series,  for any positive integer $N$, we have 
\begin{align}\label{hyper_asym}
{}_3F_2 \left( \frac{k}{2}, \frac{k+1}{2}, 1; \frac{1}{4}, \frac{3}{4}; -   \left( \frac{ y}{8 n \pi}  \right)^{2}    \right)- 1 = \sum_{j=1}^{N-1} C_j \left( \frac{y}{n} \right)^{2j} + O_{k}\left( \left( \frac{y}{n} \right)^{2N} \right),  \,\, \rm{as}\,\, y \rightarrow 0^{+}, 
\end{align}
where $C_j= (-1)^j \frac{ \left(  \frac{k}{2}\right)_j \left( \frac{k+1}{2} \right)_j}{ \left( \frac{1}{4} \right)_j \left( \frac{3}{4} \right)_j  (8\pi)^{2j}}$. Now invoke \eqref{hyper_asym} in Theorem \ref{main theorem} to derive that 
\begin{align}
y^k \sum_{n=1}^\infty  \lambda_f(n^2) \exp(-n y) &  =  \frac{\Gamma(k)}{2  \pi^2}  \sum_{j=1}^{N-1} C_j y ^{2j+1}  \sum_{n=1}^{\infty}   \frac{B_f(n) }{n^{k+2j}} + O_{k} \left( y^{2N+1} \sum_{n=1}^{\infty}   \frac{B_f(n) }{n^{k+2N}}   \right) + y^k \mathcal{R}(y)  \nonumber \\
& = \sum_{j=1}^{N-1} A_j y^{2j+1} + O_{f, k} \left( y^{2N+1}    \right) + y^k \mathcal{
R}(y), \label{asym}
\end{align}
where $A_j= \frac{\Gamma(k)}{2  \pi^2} C_j \sum_{n=1}^{\infty}   \frac{B_f(n) }{n^{k+2j}} $ are computable finite constants since the Dirichlet series associated to $B_f(n)$ is absolutely convergent for $\Re(s)>k$.  Assuming Riemann hypothesis and using the fact that the non-trivial zeros appears with conjugate pairs to write the residual term as 
\begin{align}
y^k \mathcal{R}(y)   & =  \frac{y}{2} \sum_{\substack{\rho_n= \frac{1}{2}+i t_n,\\ t_n>0}} 2 \Re \left( \frac{ \Gamma \left( \frac{\rho_n}{2} + k-1  \right) L(  \frac{\rho_n}{2} + k-1, {\rm Sym}^2(f))}{ y^{\frac{\rho_n}{2}} \zeta'(\rho_n) } \right) \nonumber \\
& = y^{3/4} \sum_{\substack{\rho_n= \frac{1}{2}+i t_n,\\ t_n>0}} b_n \cos\left( \delta_n - \frac{t_n}{2} \log(y)\right), \label{Res_exp}
\end{align}
here we have considered $b_n \exp( i \delta_n)$ as the polar representation of $\Gamma \left( \frac{\rho_n}{2} + k-1  \right) L(  \frac{\rho_n}{2} + k-1, {\rm Sym}^2(f)) \left(  \zeta'(\rho_n) \right)^{-1}$.  Employ \eqref{Res_exp} in \eqref{asym} to the complete the proof. 
\end{proof}

\section{Concluding Remarks}
We have seen that the constant terms of the automorphic form $y^k |f(z)|^2$, that is,  the Lambert series 
$
y^k \sum_{n=1}^\infty | \lambda_{f}(n)|^2 \exp(-4 \pi n y),
$
where $\lambda_f(n)$ is the $n$th Fourier coefficient of a Hecke eigenform $f(z)$ of weight $k$ over $\text{SL}_2(\mathbb{Z})$, 
has an asymptotic expansion in terms of the non-trivial zeros of $\zeta(s)$.  Recently,  authors \cite{JMS-2021} studied an asymptotic expansion of a  Lambert series associated to a Hecke eigenform and the M\"{o}bius function.  Inspired from these works, 
in the present paper,   we established an exact formula for the Lambert series
$
y^k \sum_{n=1}^\infty \lambda_{f}( n^2 ) \exp (- ny),
$
and we found that the main term can be expressed in terms of the non-trivial zeros of $\zeta(s)$,  and the error term is expressed in terms of the hypergeometric function ${}_3F_{2}(a,b,c; d; z)$.  It would be an  interesting problem to study a more general Lambert series 
$y^k \sum_{n=1}^\infty | \lambda_{f}(n) |^N \exp(- n y)$ for  $N \geq 3$. 
It would also be a challenging problem to classify automorphic forms for which constant terms will have an asymptotic expansion in terms of the non-trivial zeros of $\zeta(s)$. 

\vspace{0.5cm}

\textbf{Acknowledgements.}
The second author wants to thank  SERB for the Start-Up Research Grant SRG/2020/000144. 
The third author is greatly indebted to Prof. B. R. Shankar for his continuous support.  
He wishes to thank the National Institute of Technology Karnataka,  for financial support.

\end{document}